\documentclass[final,11pt]{amsart}
\usepackage{amsmath,amsfonts,amssymb,amscd,verbatim,comment}
\usepackage[leqno]{amsmath}
\usepackage{fullpage}
\usepackage{color}
\usepackage{array}
\usepackage{enumerate}
\usepackage{xspace}
\usepackage{bm}

\newcommand{\ba}{\mathbf{a}}
\newcommand{\bb}{\mathbf{b}}
\newcommand{\bc}{\mathbf{c}}
\newcommand{\bd}{\mathbf{d}}
\newcommand{\be}{\mathbf{e}}
\newcommand{\bx}{\mathbf{x}}
\newcommand{\by}{\mathbf{y}}
\newcommand{\bz}{\mathbf{z}}
\newcommand{\bw}{\mathbf{w}}
\newcommand{\bzero}{\mathbf{0}}

\newcommand{\Rplus}{\protect\hspace{-.1em}\protect\raisebox{.35ex}{\smaller{\smaller\textbf{+}}}}
\newcommand{\Cpp}{\mbox{C\Rplus\Rplus}\xspace}

\makeatletter
\newcommand{\leqnomode}{\tagsleft@true}
\newcommand{\reqnomode}{\tagsleft@false}
\makeatother

\numberwithin{equation}{section}
\theoremstyle{plain}
 
\newtheorem{theorem}[equation]{Theorem} 

\newtheorem{corollary}[equation]{Corollary} 
\newtheorem{lemma}[equation]{Lemma}
\newtheorem{proposition}[equation]{Proposition}
\newtheorem*{claim}{Claim}

\theoremstyle{definition}
\newtheorem{definition}[equation]{Definition}

\newtheorem{question}[equation]{Question}
\newtheorem{remark}[equation]{Remark}

\newcommand\FF{\mathbb F}
\newcommand\NN{\mathbb N}

\title{Sidon sets and 2-caps in $\mathbb{F}_3^n$}

\author[Huang]{Yixuan Huang}
\address{Wake Forest University, Department of Mathematics and Statistics, Winston-Salem, NC 27109} 
\email{huany16@wfu.edu}

\author[Tait]{Michael Tait}
\address{Carnegie Mellon University, Department of Mathematical Sciences, Pittsburgh, PA 15213} 
\email{mtait@cmu.edu}

\author[Won]{Robert Won}
\address{University of Washington, Department of Mathematics, Seattle, WA 98195} 
\email{robwon@uw.edu}


\begin{document}

\begin{abstract} 
For each natural number $d$, we introduce the concept of a $d$-cap in $\mathbb{F}_3^n$. A set of points in $\mathbb{F}_3^n$ is called a $d$-cap if, for each $k = 1, 2, \dots, d$, no $k+2$ of the points lie on a $k$-dimensional flat. This generalizes the notion of a cap in $\mathbb{F}_3^n$. We prove that the $2$-caps in $\mathbb{F}_3^n$ are exactly the Sidon sets in $\mathbb{F}_3^n$ and study the problem of determining the size of the largest $2$-cap in $\mathbb{F}_3^n$.
\end{abstract}

\maketitle

\section{Introduction}

Throughout, let $\FF_q$ denote the field with $q$ elements and let $\FF_q^n$ denote $n$-dimensional affine space over $\FF_q$. A \emph{cap} in $\FF_3^n$ is a collection of points such that no three are collinear.
Although this definition is geometric, there is an equivalent definition that is arithmetic: a set of points $C$ is a cap in $\FF_3^n$ if and only if $C$ contains no three-term arithmetic progressions.

Here, we consider natural generalizations of caps in $\FF_3^n$.
For $d \in \NN$, we call a set of points a \emph{$d$-cap} if, for each $k = 1, 2, \dots, d$, no $k+2$ of the points lie on a $k$-dimensional flat.
With this definition, a $1$-cap corresponds to the usual definition of a cap. We also remark that if $C$ is a set of points in $\FF_3^{n}$, then the points of $C$ are in general linear position if and only if $C$ is an $(n-1)$-cap.

Let $r(1,\FF_3^n)$ denote the maximal size of a $1$-cap in $\FF_3^n$. In general, it is a difficult problem to determine $r(1,\FF_3^n)$---in fact, the exact answer is known only when $n \leq 6$. Table \ref{tbl.bound} lists the best known upper and lower bounds on $r(1,\FF_3^n)$ for $n \leq 10$ \cite{V}. It is also known that in dimension $n \leq 6$, maximal $1$-caps are equivalent up to affine transformation \cite{ELS5, P4, P6}.

\begin{table}[h!]
\begin{tabular}{| c | c | c | c | c | c | c | c | c | c | c |} \hline Dimension & 1 & 2 & 3 &4 & 5& 6 & 7 & 8 & 9 & 10 \\ \hline
Lower bound &  2 & 4 & 9 & 20 & 45 & 112 & 236 & 496 & 1064 & 2240 \\ \hline
Upper bound &  2 & 4 & 9 & 20 & 45 & 112 &  291 & 771 & 2070 & 5619 \\ \hline
\end{tabular}
\caption{The best known bounds for the size of a maximal $1$-cap in $\FF_3^n$.}
\label{tbl.bound}
\end{table}

The asymptotic bounds on $r(1,\FF_3^n)$ are well-studied. In \cite{Elower}, Edel showed that 
\[\limsup_{n \to \infty} \frac{\log_3(r(1,\FF_3^n))}{n} \geq 0.724851
\]
and consequently that $r(1,\FF_3^n)$ is $\Omega(2.2174^n)$ (using Hardy-Littlewood's $\Omega$ notation). In more recent breakthrough work \cite{EG}, Ellenberg and Gijswijt (adapting a method of Croot, Lev, and Pach in \cite{CLP}) proved that $r(1,\FF_3^n)$ is $o(2.756^n)$. 

In this paper, we focus on the study of $2$-caps in $\FF_3^n$. We show that there is an equivalent arithmetic formulation of the definition of a $2$-cap. In particular, the $2$-caps in $\FF_3^n$ are exactly the Sidon sets in $\FF_3^n$, which are important objects in combinatorial number theory (we refer the interested reader to the survey \cite{OSurvey} and the references therein). Using this definition, we are able to compute the exact maximal size of a $2$-cap in $\FF_3^n$ when $n$ is even. 
We also examine $2$-caps in low dimension when $n$ is odd, in particular considering dimension $n = 3$, $5$, and $7$. 

\begin{table}[h!]
\begin{tabular}{| c | c | c | c | c | c | c | c | c | c | c |} \hline Dimension & 1 & 2 & 3 &4 & 5& 6 & 7 & 8 & $n$ even & $n$ odd \\ \hline
Lower bound &  $2$ & $3$ & $5$ & $9$ & $13$ & $27$ & $33$ & $81$ & $3^{n/2}$ & $3^{(n-1)/2} + 1$ \\ \hline
Upper bound &  $2$ & $3$ & $5$ & $9$ & $13$ & $27$ &  $47$ & $81$ & $3^{n/2}$ & $\left \lceil 3^{n/2} \right \rceil$ \\ \hline
\end{tabular}
\caption{Bounds for the size of a maximal $2$-cap in $\FF_3^n$.}
\label{tbl.second}
\end{table}

Table \ref{tbl.second} lists the bounds we obtain for the size of a maximal $2$-cap in $\FF_3^n$. The values in dimension $3$, $5$, and $7$ are given by Theorems \ref{thm.dim3} and \ref{thm.dim5}, and Proposition \ref{thm.dim7}, respectively. The bounds for even dimension follow from Theorem \ref{thm.even}. The upper bound in odd dimension $n$ follows from Proposition \ref{prop.upperbound} and the lower bound is given by adding one affinely independent point to the construction in dimension $n-1$. Knowing the exact value in even dimension also allows us to conclude that asymptotically, the maximal size of a $2$-cap in $\FF_3^n$ is $\Theta(3^{n/2})$. 

\subsection*{Acknowledgments}
The authors would like to thank W. Frank Moore for suggesting the project, as well as the anonymous referee for many helpful suggestions. Yixuan Huang was supported by a Wake Forest Research Fellowship during the summer of 2018 and Michael Tait was supported in part by NSF grant DMS-1606350.

\section{Preliminaries}

In this section, we establish basic notation, definitions, and background. The set of natural numbers is denoted $\NN = \{1,2,3,\dots\}$. Throughout, $d$ and $n$ will always denote natural numbers. An element $\ba \in \FF_3^n$ will be written as a row vector $\ba = (a_1, a_2, \dots, a_n)$ with each $a_i \in \{0,1,2\}$. 
We will sometimes order the vectors of $\FF_3^n$ lexicographically---i.e., by regarding them as ternary strings.
We use the notation $\be_1, \be_2, \ldots, \be_n$ to denote the $n$ standard basis vectors in an $n$-dimensional vector space.

A $k$-dimensional affine subspace of a vector space is called a \emph{$k$-dimensional flat}.
In particular, a $1$-dimensional flat is also called a \emph{line}. In the affine space $\FF_3^n$, every line consists of the points $\{ \ba, \ba + \bb, \ba + 2 \bb\}$ for some $\ba, \bb \in \FF_3^n$ where $\bb \neq \bzero$. Hence, the lines in $\FF_3^n$ correspond to three-term arithmetic progressions. It is easy to see that three distinct points in $\FF_3^n$ are collinear if and only if they sum to $\bzero$. Likewise, a $2$-dimensional flat is called a \emph{plane}. Any three non-collinear points determine a unique plane.
For $\ba = (a_1, a_2, \dots, a_k) \in \FF_3^k$ with $k < n$. The subset of $\FF_3^n$ whose first $k$ entries are $a_1, a_2, \ldots, a_k$ is an $(n-k)$-dimensional flat which we call \emph{the $\ba$-affine subspace} of $\FF_3^n$.

Two subsets $C$ and $D$ of a vector space are called \emph{affinely equivalent} if there exists an invertible affine transformation $T$ such that $T(C) = D$. It is clear that affine equivalence determines an equivalence relation on the power set of a vector space. Given a set of points $X$ in a vector space, its affine span is given by the set of all affine combinations of points of $X$. A set $X$ is called {\em affinely independent} if no proper subset of $X$ has the same affine span as $X$. Equivalently, $\{ \bx_0, \bx_1, \dots, \bx_n\}$ is affinely independent if and only if $\{ \bx_1 - \bx_0, \bx_2 - \bx_0, \dots, \bx_n - \bx_0\}$ is linearly independent.
 
\begin{definition}A subset $C$ of $\FF_3^n$ is called a \emph{$d$-cap} if, for each $k = 1, 2, \dots, d$, no $k+2$ points of $C$ lie on a $k$-dimensional flat. Equivalently, $C$ is a $d$-cap if and only if any subset of $C$ of size at most $d+2$ is affinely independent. 
A $d$-cap is called \emph{complete} if it is not a proper subset of another $d$-cap and is called \emph{maximal} if it is of the largest possible cardinality.
\end{definition}

As mentioned in the introduction, a $1$-cap is a classical cap. We will denote the size of a maximal $d$-cap in $\FF_3^n$ by $r(d,\FF_3^n)$.
We remark that since invertible affine transformations preserve affine independence, the image of a $d$-cap under an invertible affine transformation is again a $d$-cap.
As a warm-up, we prove some basic facts about maximal $d$-caps in $\FF_3^n$.

\begin{lemma}
\label{lem.lowdim}
We have that $r(d, \FF_3^n) \geq n+1$ with equality if $n \leq d$.
\end{lemma}
\begin{proof}The set $\{\bzero, \be_1, \dots, \be_n\}$ is an affinely independent subset of $\FF_3^n$ of size $n+1$ and hence is a $d$-cap for any $d \in \NN$. Therefore, $r(d,\FF_3^n) \geq n+1$.

Now suppose $n \leq d$. Since, by definition, a $d$-cap must be an $n$-cap, we have that $r(d,\FF_3^n) \leq r(n,\FF_3^n)$.
A maximal affinely independent set in $\FF_3^n$ has size $n+1$ so $r(n,\FF_3^n) \leq n+1$, and so $r(d,\FF_3^n) = n+1$.
\end{proof}

\begin{corollary}\label{cor.affequiv}
When $n \leq d$, all maximal $d$-caps in $\FF_3^n$ are affinely equivalent.
\end{corollary}
\begin{proof}By Lemma \ref{lem.lowdim}, when $n \leq d$, a maximal $d$-cap in $\FF_3^n$ is a maximal affinely independent set, i.e., an affine basis of $\FF_3^n$. All affine bases in an affine space are equivalent up to affine transformation.
\end{proof}

\begin{lemma}\label{lem.incdec} For fixed $d $, $r(d,\FF_3^n)$ is a non-decreasing function of $n$ and for fixed $n $, $r(d,\FF_3^n)$ is a non-increasing function of $d$.
\end{lemma}
\begin{proof}
Since $\FF_3^{n-1}$ is an affine subspace of $\FF_3^n$, a $d$-cap in $\FF_3^{n-1}$ naturally embeds as a $d$-cap in $\FF_3^n$. Hence $r(d,\FF_3^{n-1}) \leq r(d, \FF_3^n)$ so the first statement follows. The second statement follows since, by definition, a $d$-cap in $\FF_3^n$ must be a $(d-1)$-cap. Hence, $r(d-1, \FF_3^n) \geq r(d,\FF_3^n)$.
\end{proof}

\section{$2$-caps in $\FF_3^n$}

We now restrict our attention to the study of $2$-caps in $\FF_3^n$. Our first observation is that in $\FF_3^n$, the definition of a $2$-cap is equivalent to the definition of a Sidon set.

\begin{definition}Let $G$ be an abelian group. A subset $A \subseteq G$ is called a \emph{Sidon set} if, whenever $a + b = c + d$ with $a,b,c,d \in A$, the pair $(a,b)$ is a permutation of the pair $(c,d)$.
\end{definition}

\begin{theorem}\label{thm.sidon} A subset $C$ of $\FF_3^n$ is a $2$-cap if and only if it is a Sidon set.
\end{theorem}
\begin{proof}
First suppose that $C$ is not a $2$-cap. Then $C$ contains three points which are collinear or $C$ contains four points which are coplanar. 
If $C$ contains three distinct collinear points $\ba, \bb, \bc$ then $\ba + \bb + \bc = \bzero$ and hence $\ba + \bb = \bc + \bc$ so $C$ is not a Sidon set.

Suppose therefore that no three points in $C$ are collinear. Then $C$ contains four coplanar points, say $\{\ba, \bb, \bc, \bd\}$. Every set of three distinct non-collinear points in $\FF_3^n$ lies on a unique $2$-dimensional flat. In particular, the $2$-dimensional flat $F$ containing $\ba$, $\bb$, and $\bc$ is given by
\[ F =  \begin{tabular}{|c|c|c|}\hline 
$\ba$ & $\bb$ & $-\ba - \bb$  \\ \hline
$\bc$ & $-\ba + \bb + \bc$ & $\ba - \bb + \bc$ \\ \hline
$-\ba - \bc$ & $\ba + \bb - \bc$ & $-\bb - \bc$ \\ \hline
\end{tabular}
\]
and since we assumed that no three points in $C$ are collinear, we must have that $\bd = -\ba + \bb + \bc$, $\bd = \ba - \bb + \bc$ or $\bd = \ba + \bb - \bc$. In the first case, $\ba + \bd = \bb + \bc$, in the second case, $\bb + \bd = \ba + \bc$, and in the third case $\bc + \bd + \ba + \bb$. In any case, $C$ is not a Sidon set.

Conversely, suppose that $C$ is not a Sidon set. Then either $C$ contains three distinct points $\ba, \bb, \bc$ such that $\ba + \ba = \bb + \bc$, or $C$ contains four distinct points $\ba, \bb, \bc, \bd$ such that $\ba + \bb = \bc + \bd$. In the first case, $\ba + \bb + \bc = \bzero$ so $C$ contains a line. In the second case, $\bd = \ba + \bb - \bc$, so $\bd$ lies in the plane determined by $\ba$, $\bb$, and $\bc$, and hence the four points are coplanar. In either case, $C$ is not a $2$-cap.
\end{proof}

Since, in $\FF_3^n$, $2$-caps correspond to Sidon sets, we will use the terms interchangeably throughout. We obtain an upper bound on $r(2, \FF_3^n)$ by an easy counting argument (c.f. \cite[Corollary 2.2]{CRV}).

\begin{proposition}\label{prop.upperbound} For any $n \in \NN$, $r(2, \FF_3^n) \cdot (r(2, \FF_3^n) - 1 ) \leq 3^{n} - 1$. 
\end{proposition}
\begin{proof}
Suppose $C \subset \FF_3^n$ is a $2$-cap and hence, by Theorem \ref{thm.sidon}, a Sidon set. For $\ba,\bb,\bc,\bd \in C$, if $\ba-\bb = \bc-\bd$ then $\{\ba,\bd\} = \{\bc,\bb\}$ and so we have either $\ba=\bb$, or else $\ba=\bc$ and $\bb=\bd$. Therefore, the set $\{\ba-\bb : \ba,\bb\in C, \ba\neq \bb\}$ has size $|C|\left(|C|-1\right)$. Since these differences are nonzero, we have
\[
|C|\left(|C|-1\right) \leq 3^n-1. \qedhere
\]
\end{proof}

\subsection{Even dimension}
\begin{theorem}\label{thm.even} If $n$ is even, then $r(2,\FF_3^n) = 3^{n/2}$.
\end{theorem}
\begin{proof} First we will show the lower bound, $r(2, \FF_3^n) \geq 3^{n/2}$. Since $\FF_3^n$ is additively isomorphic to $\FF_3^{n/2}\times \FF_3^{n/2}$, it suffices to construct a Sidon set of size $3^{n/2}$ in $\FF_3^{n/2}\times \FF_3^{n/2}$. As vector spaces over $\FF_3$, $\FF_{3}^{n/2}$ is isomorphic to $\FF_{3^{n/2}}$, the finite field with $3^{n/2}$ elements. Hence, it suffices to construct a Sidon set of size $3^{n/2}$ in $\FF_{3^{n/2}}\times \FF_{3^{n/2}}$ This follows easily from the following claim (for a proof, see \cite[Example 1]{C}).

\begin{claim}
Let $q$ be an odd prime power and $\FF_q$ be the finite field of order $q$. Then the set $\{(x,x^2): x\in \FF_q\}$ is a Sidon set in $\FF_q \times \FF_q$.
\end{claim}

It is clear that the set $\{(x,x^2): x\in \FF_{3^{n/2}}\}$ has size $3^{n/2}$ and so we have $r(2,\FF_3^n) \geq 3^{n/2}$. For the upper bound, let $C \subset \FF_3^n$ be a $2$-cap. Since $n$ is even, $3^{n/2}$ is an integer, and if $|C| \geq 3^{n/2} + 1$, this contradicts Proposition \ref{prop.upperbound}. Therefore, $r(2, \FF_3^n) \leq 3^{n/2}$.
\end{proof}

\begin{corollary}As $n \to \infty$, $r(2, \FF_3^n)$ is $\Theta(3^{n/2})$.
\end{corollary}

The construction above can be leveraged into the following partitioning theorem.

\begin{theorem}\label{partition}
When $n$ is even, there is a partition of $\FF_3^n$ into maximal $2$-caps.
\end{theorem}

This serves as an analogue to similar results for $1$-caps in $\FF_3^n$. It is well-known that $\FF_3^3$ can be partitioned into three maximal $1$-caps of size $9$. It is possible to partition $\FF_3^2$ into a single point and two disjoint maximal $1$-caps of size $4$. Finally, \cite[Theorem 3.3]{FKMPW} shows that $\FF_3^4$ can be partitioned into a single point and four disjoint maximal $1$-caps of size $20$.
\begin{proof}[Proof of Theorem \ref{partition}]
Since translations of Sidon sets are also Sidon sets, for each $a\in \FF_{3^{n/2}}$ the set $S_a:=\{(x,x^2+a):x\in \FF_{3^{n/2}}\}$ is a maximal $2$-cap. Since $(x, x^2+a) = (y, y^2+b)$ implies $x=y$ and hence $a=b$, we have that $S_a$ and $S_b$ are disjoint for $a\not=b$. Therefore, as $a$ ranges over $\FF_{3^{n/2}}$ the sets $S_a$ cover $3^n$ points and thus there is the claimed partition.
\end{proof}

\begin{question}By Corollary \ref{cor.affequiv}, all maximal $2$-caps in $\FF_3^2$ are affinely equivalent. Is this true in $\FF_3^n$ when $n$ is even?
\end{question}

We remark that when $n = 4$, a computer program verified that all maximal $2$-caps sum to $\bzero$. If a set of nine points sums to $\bzero$ in $\FF_3^4$, then its image under any affine transformation will likewise sum to $\bzero$, so this is a necessary condition for all maximal $2$-caps in $\FF_3^4$ to be affinely equivalent.

\subsection{Odd dimension}

\begin{lemma}\label{lem.dim3} If $C = \{\ba, \bb, \bc, \bd \}$ is a $2$-cap of size four in $\FF_3^n$ then $D = \{\ba, \bb, \bc, \bd, \ba + \bb + \bc + \bd \}$ is a $2$-cap of size five.
\end{lemma}
\begin{proof}
First we note that the points of $D$ are distinct since if, without loss of generality, $\ba + \bb + \bc + \bd = \ba$, this implies that $\bb$, $\bc$, and $\bd$ are collinear, which is impossible since $C$ is a $2$-cap.

Now, suppose for contradiction that $D$ is not a $2$-cap, so there exist some $\bx, \by, \bz, \bw \in D$ with $\bx + \by = \bz + \bw$. Since $C$ is a $2$-cap, we may assume that $\bx = \ba + \bb + \bc + \bd$. Without loss of generality, we then have that one of the following occurs:
\begin{enumerate}
\item $(\ba + \bb + \bc + \bd) + \ba = \bb + \bc$. Then $\ba = \bd$, which is impossible since $C$ has size four.

\item $(\ba + \bb + \bc + \bd) + \ba = 2\bb$. Then $\ba + \bb = \bc + \bd$,  which is impossible since $C$ is a $2$-cap.

\item $2(\ba + \bb + \bc + \bd) = \bb + \bc$. Then $\ba + \bd = \bb + \bc$,  which is impossible since $C$ is a $2$-cap.

\item $2(\ba + \bb + \bc + \bd) = 2 \ba$. Then $\bb$, $\bc$, and $\bd$ are collinear,  which is impossible since $C$ is a $2$-cap.
\end{enumerate}
Hence, $D$ is a $2$-cap.
\end{proof}

\begin{theorem} \label{thm.dim3} In $\FF_3^3$, a maximal $2$-cap has size $5$, that is, $r(2, \FF_3^3) = 5$. Further, all complete $2$-caps are maximal and all maximal $2$-caps are affinely equivalent.
\end{theorem}
\begin{proof}
Since $\{\bzero, \be_1, \be_2, \be_3\}$ is an affinely independent set in $\FF_3^3$, by Lemma \ref{lem.dim3}, $\{\bzero, \be_1, \be_2, \be_3, \be_1 + \be_2 + \be_3 \}$ is a $2$-cap in $\FF_3^3$. Hence, $r(2,\FF_3^3) \geq 5$. But by Proposition \ref{prop.upperbound}, $r(2,\FF_3^3) < 6$ and hence $r(2, \FF_3^3) = 5$.

Let $C$ be any complete $2$-cap in $\FF_3^3$. Since $\FF_3^3$ is a three-dimensional affine space, if $|C| \leq 3$, then $\FF_3^3$ contains a point which is affinely independent from the points of $C$, so $C$ cannot be complete. Hence, $|C| \geq 4$. But if $|C| = 4$ then by Lemma \ref{lem.dim3}, $C$ is not complete. Hence, $|C| = 5$, and any complete $2$-cap in $\FF_3^3$ is already maximal.

For the final claim, suppose $C$ is a maximal $2$-cap in $\FF_3^3$. Pick any four points in $C$. Since these points are affinely independent, there exists an invertible affine transformation mapping these points to the set $\{\bzero, \be_1, \be_2, \be_3\}$. Hence, we need only show that all maximal $2$-caps containing $\{\bzero, \be_1, \be_2, \be_3\}$ are affinely equivalent.

It is easy to verify that there are exactly five such maximal $2$-caps, namely:
\begin{enumerate}
\item $C_1 =\{ \bzero, \be_1, \be_2, \be_3, (1,1,1)\}$,
\item $C_2 = \{\bzero, \be_1, \be_2, \be_3, (1,2,2)\}$,
\item $C_3 = \{\bzero, \be_1, \be_2, \be_3, (2,1,2)\}$,
\item $C_4 = \{\bzero, \be_1, \be_2, \be_3, (2,2,1)\}$, and
\item $C_5 = \{\bzero, \be_1, \be_2, \be_3, (2,2,2)\}$.
\end{enumerate}
It suffices to exhibit an invertible affine transformation $T_{i}$ mapping $C_1$ to $C_i$ for $i = 2, 3, 4, 5$. We provide these $T_i$ explicitly, writing $T_{i}(\bx) = A_i \bx + \bb_i$ for an invertible matrix $A_i$ and $\bb_i \in \FF_3^3$. 
\begin{enumerate}
\item $A_2=\begin{bmatrix}
1 & 0 & 0 \\
2 & 1 & 0 \\
2 & 0 & 1
\end{bmatrix}$ and $\bb_2 = \begin{bmatrix}0 \\ 0 \\ 0 \end{bmatrix}$,
\item 
$A_3=\begin{bmatrix}
1 & 0 & 2 \\
0 & 0 & 1 \\
0 & 1 & 2
\end{bmatrix}$ and $\bb_3 = \begin{bmatrix}0 \\ 0 \\ 0 \end{bmatrix}$,
\item $A_4=\begin{bmatrix}
1 & 0 & 2 \\
0 & 1 & 2 \\
0 & 0 & 1
\end{bmatrix}$ and $\bb_4 = \begin{bmatrix}0 \\ 0 \\ 0 \end{bmatrix}$, and
\item$A_5=\begin{bmatrix}
2 & 1 & 1 \\
1 & 2 & 1 \\
1 & 1 & 2
\end{bmatrix}$ and 
$\bb_5 = \begin{bmatrix}
2 \\ 2 \\2
\end{bmatrix}$. \qedhere
\end{enumerate}
\end{proof}

\begin{theorem}\label{thm.dim5} A maximal $2$-cap in $\FF_3^5$ has size $13$, that is, $r(2,\FF_3^5) = 13$.
\end{theorem}
\begin{proof}
Let $C$ be a maximal $2$-cap in $\FF_3^5$. By Theorem \ref{thm.even}, $r(2,\FF_3^4) = 9$ so by Lemma \ref{lem.incdec} we may assume that $|C| \geq 9$. We will apply a sequence of affine transformations to $C$ to conclude that lexicographically, the first points in $C$ are $\{\bzero, \be_5, \be_4, \be_3, \be_3 +\be_4 + \be_5, \be_2\}$ or $\{\bzero, \be_5, \be_4, \be_3, \be_2\}$.

Given any four affinely independent points, there exists an invertible affine transformation mapping them to $\bzero$, $\be_5$, $\be_4$, and $\be_3$, so without loss of generality we may assume that $C$ contains the subset $\{\bzero, \be_5, \be_4, \be_3\}$. These points all lie in the $(0,0)$-affine subspace of $\FF_3^5$. Since $r(2,\FF_3^3) = 5$, the $(0,0)$-affine subspace contains four points or five points of $C$. If it contains five points, then by Theorem \ref{thm.dim3}, we may apply an affine transformation (using a block matrix) and assume that the fifth point is $\be_3 + \be_4 + \be_5$. 

Consider any other point $\ba \in C$. Since $\ba$ is not in the $(0,0)$-affine subspace of $\FF_3^5$, therefore $\{\bzero, \be_5, \be_4, \be_3, \ba\}$ is an affinely independent set so there exists an affine transformation $T$ fixing $\bzero$, $\be_5$, $\be_4$, and $\be_3$ and mapping $\ba$ to $\be_2$. Notice that if $T$ is given by multiplication by the invertible matrix $A$ followed by addition by $\bb \in \FF_3^5$, we have that
\[T(\be_3 + \be_4 + \be_5)= A(\be_3 + \be_4 + \be_5) + \bb = T(\bzero) + T(\be_3) + T(\be_4) + T(\be_5) = \be_3 + \be_4 + \be_5
\]
so $T$ fixes $\be_3 + \be_4 + \be_5$.

Hence, up to affine equivalence, we may assume that the lexicographically earliest points in $C$ are $\{\bzero, \be_5, \be_4, \be_3, \be_3 +\be_4 + \be_5, \be_2\}$ or $\{\bzero, \be_5, \be_4, \be_3, \be_2\}$.
A computer program was used to enumerate all possible complete $2$-caps beginning with these sets of points. This verified that $r(2,\FF_3^5) = 13$. The \Cpp code for the program is available on the third author's professional website. 
\end{proof}

\begin{remark}The maximal $2$-cap in $\FF_3^5$ that is lexicographically earliest is explicitly given by the points: $(0, 0, 0, 0, 0)$, $(0, 0, 0, 0, 1)$, $(0, 0, 0, 1, 0)$, $(0, 0, 1, 0, 0)$, $(0, 0, 1, 1, 1)$, $(0, 1, 0, 0, 0)$, $(0, 1, 1, 1, 2)$, $(0, 2, 1, 2, 0)$, $(0, 2, 2, 1, 2)$, $(1, 0, 0, 0, 0)$, $(1, 0, 1, 2, 1)$, $(2, 0, 1, 0, 2)$, $(2, 2, 0, 2, 2)$.
\end{remark}

We conclude by giving bounds on $r(2,\FF_3^7)$.

\begin{proposition} 
\label{thm.dim7}
One has that $33 \leq r(2,\FF_3^7) \leq 47$.
\end{proposition}
\begin{proof}The upper bound on $r(2,\FF_3^7)$ is a consequence of Proposition \ref{prop.upperbound}. For the lower bound, we constructed a $2$-cap of size $33$ by first embedding a maximal $2$-cap in $\FF_3^6$ as a $2$-cap $C$ of size $27$ in $\FF_3^7$. We then used a computer program to enumerate all complete $2$-caps containing $C$ as a subset. The largest of these complete $2$-caps has size $33$. The lexicographically earliest one is given by the points:
\fbox{\begin{minipage}[t]{\textwidth}
$(0, 0, 0, 0, 0, 0, 0)$, $(0, 0, 0, 1, 0, 0, 1)$, $(0, 0, 0, 2, 0, 0, 1)$, $(0, 0, 1, 0, 1, 0, 0)$, $(0, 0, 1, 1, 1, 2, 1)$, $(0, 0, 1, 2, 1, 1, 1)$, $(0, 0, 2, 0, 1, 0, 0)$, $(0, 0, 2, 1, 1, 1, 1)$, $(0, 0, 2, 2, 1, 2, 1)$, $(0, 1, 0, 0, 1, 2, 0)$, $(0, 1, 0, 1, 0, 2, 1)$, $(0, 1, 0, 2, 2, 2, 1)$, $(0, 1, 1, 0, 2, 1, 1)$, $(0, 1, 1, 1, 1, 0, 2)$, $(0, 1, 1, 2, 0, 2, 2)$, $(0, 1, 2, 0, 2, 0, 2)$, $(0, 1, 2, 1, 1, 1, 0)$, $(0, 1, 2, 2, 0, 2, 0)$, $(0, 2, 0, 0, 1, 2, 0)$, $(0, 2, 0, 1, 2, 2, 1)$, $(0, 2, 0, 2, 0, 2, 1)$, $(0, 2, 1, 0, 2, 0, 2)$, $(0, 2, 1, 1, 0, 2, 0)$, $(0, 2, 1, 2, 1, 1, 0)$, $(0, 2, 2, 0, 2, 1, 1)$, $(0, 2, 2, 1, 0, 2, 2)$, $(0, 2, 2, 2, 1, 0, 2)$, $(1, 0, 0, 0, 0, 0, 0)$, $(1, 0, 0, 0, 0, 0, 1)$, $(2, 0, 0, 1, 0, 2, 0)$, $(2, 0, 0, 1, 1, 0, 1)$, $(2, 0, 0, 1, 1, 1, 2)$, and $(2, 0, 0, 1, 1, 2, 2)$.
\end{minipage}}
\qedhere
\end{proof}

\bibliographystyle{plain}
\bibliography{biblio}

\end{document}